\documentclass[final,leqno]{siamltex}
\usepackage{amsmath}
\usepackage{amssymb}
\usepackage{graphicx}
\usepackage[notcite,notref]{showkeys}
\usepackage{tikz}
 \numberwithin{dummy}{section}

\newtheorem{algorithm}{Algorithm}

\newcommand{\bx}{{\bf x}}

\newcommand{\bv}{{\bf v}}

\def\T{{\mathcal T}}
\def\E{{\mathcal E}}

\def\pT{{\partial T}}
\def\l{{\langle}}
\def\r{{\rangle}}

\def\T{{\mathcal T}}
\def\E{{\mathcal E}}

\def\bn{{\bf n}}

\def\bbQ{\mathbb{Q}}

\def\3bar{{|\hspace{-.02in}|\hspace{-.02in}|}}

\title{A conforming discontinuous Galerkin finite element method}
\author{Xiu Ye\thanks{Department of
Mathematics, University of Arkansas at Little Rock, Little Rock, AR
72204 (xxye@ualr.edu). This research was supported in part by
National Science Foundation Grant DMS-1620016.}
\and
Shangyou Zhang\thanks{Department of
Mathematical Sciences, University of Delaware, Newark, DE 19716 (szhang@udel.edu).}
}
\begin{document}  \baselineskip=16pt\parskip=10pt
\maketitle

\begin{abstract}\baselineskip=13pt
A new finite element method with discontinuous approximation is introduced for solving second order elliptic problem. Since this method combines the features of both conforming  finite element method and discontinuous Galerkin (DG) method, we call it conforming DG method. While using DG finite element space, this conforming DG method maintains the features of the conforming finite element method such as simple formulation and strong enforcement of boundary condition.
Therefore, this finite element method has the flexibility of using discontinuous approximation  and simplicity in formulation of  the conforming finite element method.
Error estimates of optimal order are established for the corresponding discontinuous finite element approximation in both a discrete $H^1$ norm
and the $L^2$ norm. Numerical results are presented to confirm the theory.
\end{abstract}


\pagestyle{myheadings}

\section{Introduction}
For the sake of clear presentation, we consider Poisson equation with Dirichlet boundary condition in two dimension as our model problem. This conforming DG method can be extended to solve other elliptic problems.
The Poisson problem  seeks an unknown function $u$ satisfying
\begin{eqnarray}
-\Delta u&=&f,\quad \mbox{in}\;\Omega,\label{pde}\\
u&=&g,\quad\mbox{on}\;\partial\Omega,\label{bc}
\end{eqnarray}
where $\Omega$ is a polytopal domain in $\mathbb{R}^2$.

Researchers started to use discontinuous approximation in finite element procedure in the early 1970s \cite{Babu73, DoDu76,ReHi73, Whee78}. Local discontinuous Galerkin methods were introduced in \cite{cs1998}.   Then a paper \cite{abcm} in 2002 provides a unified analysis of discontinuous Galerkin (DG) finite element methods for Poisson equation.
Since then, many new finite element methods with discontinuous approximations have been developed such as
hybridizable discontinuous Galerkin (HDG) method \cite{cgl}, mimetic finite differences method \cite{Lipnikov2011},
hybrid high-order (HHO)  method
\cite{de}, virtual element (VE) method \cite{vbcmmr}, weak Galerkin (WG) method \cite{wy}  and references therein.

The weak form of the problem (\ref{pde})-(\ref{bc}) is given as follows: find $u\in H^1(\Omega)$
such that $u=g$ on $\partial\Omega$ and
\begin{eqnarray}
(\nabla u,\nabla v)=(f,v)\quad \forall v\in
H_0^1(\Omega).\label{weakform}
\end{eqnarray}

The conforming finite element method for the problem (\ref{pde})-(\ref{bc}) keeps the same simple form as in (\ref{weakform}). However,
when discontinuous approximation is used, finite element formulations tend to be more complex than (\ref{weakform}) to ensure  connection of discontinuous function across element boundary. For example, the following is the formulation for the symmetric interior penalty discontinuous Galerkin (IPDG) method for the Poisson equation (\ref{pde}) with homogeneous boundary condition: find $u_h\in V_h$ such that for all $v_h\in V_h$,
\[
\sum_{T\in\T_h}(\nabla u_h, \nabla v_h)_T-\sum_{e\in\E_h} \int_e
   \Big( \{\nabla u_h\}[v_h]+\{\nabla v_h\}[u_h] -\alpha h_e^{-1}[u_h][v_h]\Big)=(f,v_h),
\]
where $\alpha$ is called a penalty parameter that needs to be tuned.

A first order weakly over-penalized symmetric interior penalty method is proposed in \cite{bos} aiming for
  simplifying the above IPDG formulation by eliminating the two nonsymmetric middle terms:
  find $u_h\in V_h$  such that for all $v_h\in V_h$,
\begin{eqnarray*}
\sum_{T\in\T_h}(\nabla u_h,\nabla v_h)_T+\alpha\sum_{e\in\E_h}h_e^{-3}(\Pi_0[u_h],\;\Pi_0[v_h])_e=(f,v_h),
\end{eqnarray*}
where $\Pi_0$ is the $L^2$ projection to the constant space
   and $\alpha$ is a positive number. The price paid for a simpler formulation is a worse condition number for the
   resulting  system of linear   equations.

In this paper, we propose a new conforming DG method using the same finite element space
   used in the IPDG method for any polynomial degree $k\ge 1$
   but having a simple symmetric and positive definite system:
  find $u_h\in V_h$ satisfying $u_h=I_hg$ on $\partial\Omega$ and
\begin{equation}\label{intro}
(\nabla_w u_h,\nabla_w v_h)=(f,v_h)\quad \forall v_h\in V_h^0,
\end{equation}
where $\nabla_w$ is called weak gradient introduced in the weak Galerkin finite element method \cite{wy, wymix}. It follows from (\ref{intro}) that the conforming DG method can be obtained from the conforming formulation simply by replacing $\nabla$ by $\nabla_w$ and  enforcing the boundary condition strongly.   The simplicity of the conforming DG formulation will ease the complexity for implementation of DG methods. The computation of weak gradient $\nabla_w v$ is totally local. Optimal convergence rates for the conforming DG approximation are obtained in a
   discrete $H^1$ norm and in the $L^2$ norm.
 This new conforming DG method is tested numerically for $k=1,2,3,4$ and $5$, and the results confirm the theory.

\section{Finite Element Method}\label{Section:mwg}

In this section, we will introduce the conforming DG method.
For any given polygon $D\subseteq\Omega$, we use the standard
definition of Sobolev spaces $H^s(D)$ with $s\ge 0$. The associated inner product,
norm, and semi-norms in $H^s(D)$ are denoted by
$(\cdot,\cdot)_{s,D}$, $\|\cdot\|_{s,D}$, and $|\cdot|_{s,D}$, respectively. When $s=0$, $H^0(D)$ coincides with the space
of square integrable functions $L^2(D)$. In this case, the subscript
$s$ is suppressed from the notation of norm, semi-norm, and inner
products. Furthermore, the subscript $D$ is also suppressed when
$D=\Omega$.

Let ${\cal T}_h$ be a triangulation of the domain $\Omega$ with mesh
size $h$ that consists of triangles. Denote by
${\cal E}_h$ the set of all edges in ${\cal T}_h$, and let
${\cal E}_h^0={\cal E}_h\backslash\partial\Omega$ be the set of all
interior edges.

We define the average and the
jump on edges for a scalar-valued function $v$. For an
interior edge $e\in {\cal E}_h^0$, let $T_1$ and $T_2$ be two triangles
sharing $e$.
Let $\bn_1$ and $\bn_2$ be the two unit outward normal
  vectors on $e$, associated with ${T_1}$ and ${T_2}$, respectively.
Define the average $\{\cdot\}$ and the jump $[\cdot]$ on $e$ by
\begin{equation}\label{avg}
\{v\}=\frac12(v|_{T_1}+v|_{T_2}) \quad\hbox{and \ }  [v]=v|_{T_1}\bn_1+ v|_{T_2}\bn_2,
\end{equation}
 respectively.
If $e$ is a boundary edge, then
\begin{equation}\label{avgb}
\{v\}=v,\quad  [ v]= v \bn.
\end{equation}

For simplicity, we adopt the following notations,
\begin{eqnarray*}
(v,w) &=&(v,w)_{\T_h} = \sum_{T\in\T_h}(v,w)_T=\sum_{T\in\T_h}\int_T vw d\bx,\\
 \l v,w\r_{\partial\T_h}&=&\sum_{T\in\T_h} \l v,w\r_\pT=\sum_{T\in\T_h} \int_\pT vw ds.
\end{eqnarray*}

First we define two discontinuous finite element spaces for $k\ge 1$,
\begin{equation}\label{Vh}
V_h=\left\{ v\in L^2(\Omega):\ v|_{T}\in
P_{k}(T),\;\; T\in\T_h\right\},
\end{equation}
and
\begin{equation}\label{Vh0}
V_h^0=\left\{ v\in V_h:\ v=0 \;{\rm on}\;\partial\Omega\right\}.
\end{equation}


\smallskip

\begin{algorithm}
A conforming DG finite element method for the problem (\ref{pde})-(\ref{bc})
seeks $u_h\in V_h$  satisfying $u_h=I_h g$ on $\partial\Omega$ and
\begin{eqnarray}
(\nabla_w u_h,\nabla_w v)_{\T_h} &=&(f,\;v)\quad\forall v\in V_h^0,\label{mwg}
\end{eqnarray}
where $I_h$ is the $k$th order Lagrange interpolation.
\end{algorithm}

Next we will discuss how to compute weak gradient $\nabla_wu_h$ and $\nabla_w v$ in (\ref{mwg}).  The concept of weak gradient $\nabla_w$ was first introduced in \cite{wy,wymix} for weak functions in WG methods and was modified in \cite{mwg, mwg1} for the functions in $V_h$ in (\ref{Vh}) as follows.
For a given $T\in\T_h$ and a function $v\in V_h$, the weak gradient $\nabla_wv\in RT_k(T)$ on $T$ is the unique solution of the following equation,
\begin{equation}\label{wg}
(\nabla_{w} v, \tau)_T = -(v,\nabla\cdot \tau)_T+ \langle \{v\},
\tau\cdot\bn\rangle_{\partial T},\qquad \forall \tau\in RT_k(T),
\end{equation}
where $RT_k(T)=[P_k(T)]^2+{\bf x}P_k(T)$ and $\{v\}$ is defined in (\ref{avg}) and (\ref{avgb}). The weak gradient $\nabla_w$ is a local operator computed at each element.

\section{Well Posedness}

We start this section by introducing two semi-norms $\3bar v\3bar$ and  $\|v\|_{1,h}$
   for any $v\in V_h$ as follows:
\begin{eqnarray}
\3bar v\3bar^2 &=& \sum_{T\in\T_h}(\nabla_wv,\nabla_wv)_T, \label{norm2}\\
\|v\|_{1,h}^2&=&\sum_{T\in \T_h}\|\nabla v\|_T^2+\sum_{e\in\E_h^0}h_e^{-1}\|[v]\|_{e}^2.\label{norm3}
\end{eqnarray}

The following norm equivalences is proved in Lemma 3.2 \cite{mwwy} with $v_0=v$ and $v_b=\{v\}$ that there exist two constants $C_1$ and $C_2$ independent of $h$ such that
\begin{equation}\label{norm-e}
C_1\|v\|_{1,h}\le \3bar v\3bar\le C_2 \|v\|_{1,h},\quad\forall v\in V_h^0.
\end{equation}

\begin{lemma}
The semi-norm $\3bar\cdot\3bar$ defined in (\ref{norm2}) is a norm in $V_h^0$.
\end{lemma}

\smallskip

\begin{proof}
We only need to prove $v=0$ if $\3bar v\3bar=0$ for all $v\in V_h^0$. Let $v\in V_h^0$ and $\3bar v\3bar=0$. By (\ref{norm-e}), we have $\|v\|_{1,h}=0$ which implies that $\nabla v=0$ in each $T\in\T_h$ and $[v]=0$ on $e\in\E_h^0$.  $\nabla v=0$ on $T$ implies that $v$ is a constant on each $T$. $[v]=0$ on $e$ means that $v$ is continuous.
Thus $v$ is a global constant on the whole domain.
 With $v=0$ on $\partial\Omega$, we conclude $v=0$. This completes the proof of the lemma.
\end{proof}

The well posedness of the conforming DG method (\ref{mwg}) follows immediately from the above lemma.

\section{Error Equation}
In this section, we will derive an error equation which will be used in the convergence analysis. First we define
$H({\rm div};\Omega)$ space as the set of vector-valued
functions on $\Omega$ which, together with their divergence, are
square integrable; i.e.,
\[
H({\rm div}; \Omega)=\left\{ \bv: \ \bv\in [L^2(\Omega)]^2,
\nabla\cdot\bv \in L^2(\Omega)\right\}.
\]

Define  an interpolation operator $\bbQ_h$ for $\tau\in H({\rm div},\Omega)$ (see \cite{bf}) such that
$\bbQ_h\tau\in H({\rm div},\Omega)$,  $\bbQ_h\tau \in RT_k(T)$ on each $T\in {\cal T}_h$, and  satisfies:
\begin{equation}\label{key}
(\nabla\cdot\tau,\;v)_T=(\nabla\cdot\bbQ_h\tau,\;v)_T  \qquad
\forall v\in P_k(T).
\end{equation}

\begin{lemma}
For any $\tau\in  H({\rm div},\Omega)$,
\begin{equation}\label{key1}
-(\nabla\cdot\tau, \;v)_{\T_h}=(\bbQ_h\tau, \;\nabla_w v)_{\T_h} \quad\forall v\in V_h^0.
\end{equation}
\end{lemma}
\begin{proof}
Since $\{v\}=v=0$ on $\partial\Omega$ and $\bbQ_h\tau\in H({\rm div},\Omega)$, then
\begin{equation}\label{ee10}
  \langle \bbQ_h\tau\cdot\bn, \{v\}\rangle_{\partial \T_h}=0.
\end{equation}
It follows from (\ref{key}), (\ref{wg}) and (\ref{ee10}) that
\begin{eqnarray*}
-(\nabla\cdot\tau, \;v)_{\T_h}&=&-(\nabla\cdot\bbQ_h\tau, \;v)_{\T_h}\\
&=&-(\nabla\cdot\bbQ_h\tau, \;v)_{\T_h}+\l \{v\}, \bbQ_h\tau\cdot\bn\r_{\partial \T_h}\\
&=&(\bbQ_h\tau, \;\nabla_w v)_{\T_h},
\end{eqnarray*}
which proves the lemma.
\end{proof}

Define a continuous finite element space  $\tilde V_h$, a subspace of $V_h$, by
\begin{align}\label{tVh}
\tilde V_h = \{ v \in H^1(\Omega): \; v|_T \in P_k(T), \; \forall T\in\mathcal{T}_h\}.
\end{align}

\begin{lemma}\label{ge}
For any $v\in \tilde V_h$,
\begin{eqnarray*}
\nabla_wv=\nabla v.
\end{eqnarray*}
\end{lemma}
\begin{proof}
By the definition of the weak gradient (\ref{wg}) and integration by parts, we have for any $\tau\in RT_k(T)$,
\begin{eqnarray*}
(\nabla_{w} v, \tau)_T &=& -(v,\nabla\cdot \tau)_T+ \langle \{v\},
\tau\cdot\bn\rangle_{\partial T}\\
&=&-(v,\nabla\cdot \tau)_T+ \langle v,
\tau\cdot\bn\rangle_{\partial T}\\
&=&(\nabla v, \tau)_T,
\end{eqnarray*}
which implies
\begin{eqnarray*}
(\nabla_{w} v-\nabla v, \tau)_T=0, \quad\forall \tau\in RT_k(T).
\end{eqnarray*}
Since $\nabla_{w} v-\nabla v\in RT_k(T)$, letting $\tau =\nabla_{w} v-\nabla v$ in the above equation gives
\[
\|\nabla_{w} v-\nabla v\|_T^2=0,
\]
which proves the lemma.
\end{proof}

Let $e_h=I_hu-u_h$. Obviously, $e_h\in V_h^0$. Recall that $I_hu$  is the $k$th order Lagrange interpolation of $u$ and then $I_hu\in \tilde V_h$. By Lemma \ref{ge}, we have
\begin{eqnarray}
\nabla_wI_hu&=&\nabla I_hu.\label{ee11}
\end{eqnarray}

\smallskip

\begin{lemma}\label{Lemma:error-equation}
Let $e_h=I_hu-u_h$ be the error of the finite element solution arising from (\ref{mwg}). Then we have
\begin{eqnarray}
(\nabla_we_h,\; \nabla_wv)_{\T_h}&=&l(u,v),\quad\forall v\in V_h^0,\label{ee}
\end{eqnarray}
where
\begin{eqnarray}
l(u,v)&=&(\nabla I_hu-\bbQ_h\nabla u,\; \nabla_wv)_{\T_h}.\label{ll}
\end{eqnarray}
\end{lemma}

\begin{proof} Testing the equation (\ref{pde}) by $v\in V_h^0$ gives
\begin{equation}\label{m1}
-(\nabla\cdot\nabla u, v)=(f,v).
\end{equation}
It follows from (\ref{key1}) that
\begin{eqnarray}
(\bbQ_h \nabla u,\nabla_w v)_{\T_h}&=&(f,v).\label{j1}
\end{eqnarray}

Adding $(\nabla_w I_hu,\nabla_w v)_{\T_h}$ to the both sides of the equation (\ref{j1}) and using (\ref{ee11})  yield
\begin{eqnarray}
(\nabla_w I_hu,\nabla_w v)_{\T_h}&=&(f,v)+(\nabla I_hu-\bbQ_h\nabla u,\nabla_w v)_{\T_h}.\label{j2}
\end{eqnarray}
The difference of (\ref{j2}) and (\ref{mwg}) gives (\ref{ee}).
We have proved the lemma.
\end{proof}

\section{Error Estimates}\label{Section:error-analysis}
In this section, we shall establish optimal order error estimates
for $u_h$ in  a discrete $H^1$ norm and the $L^2$ norm.

\subsection{An Estimate in a Discrete $H^1$ Norm}
We start this subsection by bounding  the term $l(u,v)$ defined in (\ref{ll}).

\smallskip

\begin{lemma}
Let $u\in H^{k+1}(\Omega)$  and
$v\in V_h^0$.  Then, the following estimate holds,
\begin{eqnarray}
|l(u, v)|&\le& Ch^{k}|u|_{k+1}\3bar v\3bar.\label{mmm1}
\end{eqnarray}
\end{lemma}
\begin{proof}
Using the Cauchy-Schwarz inequality and the definition of $I_h$ and $\bbQ_h$, we have
\begin{eqnarray*}
l(u,v)&=&(\nabla I_h u-\bbQ_h(\nabla u),\nabla_w v)_{\T_h}\\
&\le&\sum_{T\in \T_h} \|\nabla I_h u-\bbQ_h(\nabla u)\|_T\|\nabla_w v\|_T\\
&\le&\Big(\sum_{T\in \T_h} \|\nabla I_h u-\bbQ_h(\nabla u)\|_T^2\Big)^{1/2}
   \Big(\sum_{T\in \T_h} \|\nabla_w v\|_T^2\Big)^{1/2}\\
&\le&\Big(\sum_{T\in \T_h} \|\nabla I_h u-\nabla u\|_T^2+\|\nabla u-\bbQ_h(\nabla u)\|_T^2\Big)^{1/2}\3bar v\3bar\\
&\le& Ch^k|u|_{k+1}\3bar v\3bar,
\end{eqnarray*}
which proves the lemma.
\end{proof}

\smallskip

\begin{theorem}\label{thm1}
Let $u_h\in V_h$ be the  finite element solution of (\ref{mwg}). Assume the exact solution $u\in H^{k+1}(\Omega)$. Then,
there exists a constant $C$ such that
\begin{equation}\label{err1}
\3bar u_h-I_h u\3bar \le Ch^{k}|u|_{k+1}.
\end{equation}
\end{theorem}
\begin{proof}
Letting $v=e_h$ in (\ref{ee}) gives
\begin{eqnarray}
\3bar e_h\3bar^2&=&l(u, e_h).\label{main}
\end{eqnarray}
Using (\ref{mmm1}), we arrive
\[
\3bar e_h\3bar^2 \le Ch^{k}|u|_{k+1}\3bar e_h\3bar,
\]
which completes the proof.
\end{proof}

\subsection{An Estimate in the $L^2$ Norm}
In this subsection, we will derive the error estimate for $u_h$ in the $L^2$ norm.
First we define $\tilde V_h^0$ a subspace of $\tilde V_h$ in (\ref{tVh}) as
\begin{align}
\tilde V_h^0 = \{ v\in \tilde V_h : v|_{\partial \Omega}=0\}.
\end{align}
Let $\tilde u_h\in \tilde V_h $ be the conforming finite element solution such that
    $\tilde u_h=I_hg$ on $\partial\Omega$ and satisfies
\begin{align}\label{l240}
(\nabla \tilde u_h, \nabla v) = (f,  v)\quad \forall v \in \tilde V_h^0 .
\end{align}

Since $\tilde V_h^0\subset  V_h^0$, by Lemma \ref{ge}, (\ref{mwg}) and (\ref{l240}), we have
\begin{align} \label{l230}
(\nabla_w  u_h - \nabla\tilde u_h, \nabla v) = 0, \quad \forall  v \in \tilde V_h^0.
\end{align}

Consider the dual problem: seek $\Phi\in H_0^1(\Omega)$ satisfying
\begin{eqnarray}
-\nabla\cdot ( \nabla\Phi)&=&u_h- \tilde u_h\quad
\mbox{in}\;\Omega.\label{dual}
\end{eqnarray}
Assume that the following $H^{2}$-regularity holds
\begin{equation}\label{reg}
\|\Phi\|_2\le C\|u_h- \tilde u_h\|.
\end{equation}

Now we are ready to derive the $L^2$ error estimate.

\smallskip

\begin{theorem} Let $u_h\in V_h$ be the finite element solution of (\ref{mwg}). Assume that the
exact solution $u\in H^{k+1}(\Omega)$ and that (\ref{reg}) holds true.
 Then, there exists a constant $C$ such that
\begin{equation}\label{err2}
\|u-u_h\| \le Ch^{k+1}|u|_{k+1}.
\end{equation}
\end{theorem}
\begin{proof}
By the triangle inequality, we have
\begin{equation}\label{ti}
\|u-u_h\|\le \|u- \tilde u_h\|+ \|u_h- \tilde u_h\|.
\end{equation}
The definition of $\tilde u_h$ implies
\begin{equation}\label{l231}
\|u- \tilde u_h\|\le Ch^{k+1}|u|_{k+1}.
\end{equation}
Next we will estimate $\|u_h- \tilde u_h\|$.
Let $\Phi_h\in V_h^0$ be the conforming DG approximation to the problem (\ref{dual}) satisfying
\begin{equation}\label{l232}
(\nabla_w\Phi_h,\nabla_w v) =(u_h- \tilde u_h, v),\quad\forall v\in V_h^0.
\end{equation}
Letting $v=u_h-\tilde u_h\in V^0_h$ in (\ref{l232}) and using Lemma \ref{ge} and (\ref{l230}),  we have,
\begin{align*}
\|u_h - \tilde u_h\|^2 &=(\nabla_w \Phi_h, \nabla_w(u_h- \tilde u_h))_{\T_h}
                       = ( \nabla_w \Phi_h,  \nabla_w u_h- \nabla \tilde u_h)_{\T_h} \\
               &=(  \nabla_w (\Phi_h - I_h \Phi),   \nabla_w u_h- \nabla \tilde u_h)_{\T_h}.
\end{align*}
By the Cauchy-Schwartz inequality, (\ref{err1}) and (\ref{reg}), then
\begin{align*}
\|u_h - \tilde u_h\|^2
    &\le \3bar \Phi_h - I_h \Phi\3bar \; (\3bar u_h- I_h u \3bar+\|\nabla( I_h u - \tilde u_h)\|) \\
    &\le C h |\Phi|_2 h^{k} |u|_{k+1} \\
    &\le Ch^{k+1} |u|_{k+1}\| u_h-\tilde u_h \|,
\end{align*}
which implies
\begin{equation}\label{l233}
\|u_h - \tilde u_h\|\le Ch^{k+1}|u|_{k+1}.
\end{equation}
Combining (\ref{l231}) and (\ref{l233}) with (\ref{ti}), we have proved the theorem.
\end{proof}

\section{Numerical Example}

 We solve the following Poisson equation on the unit square:
\begin{align} \label{s1} -\Delta u = 2\pi^2 \sin\pi x\sin \pi y,  \quad (x,y)\in\Omega=(0,1)^2,
\end{align} with the boundary condition $u=0$ on $\partial \Omega$.

In computation, the first grid consists of two unit right triangles cutting from the unit square by a forward
  slash.   The high level grids are the half-size refinement of the previous grid.
We apply $P_k$ finite element methods $V_h$ and list the error and the order of convergence in
  the following  table.
The numerical results confirm the convergence theory.

\begin{table}[h!]
  \centering \renewcommand{\arraystretch}{1.3}
  \caption{Error profiles and convergence rates for \eqref{s1} }\label{t1}
\begin{tabular}{c|cc|cc}
\hline 
level & $\|u_h- u\| $  &rate & $\3bar u_h-I_h u\3bar $ &rate \\
\hline
  &\multicolumn{4}{c}{by $P_1$ elements} \\ \hline
 6&   0.7280E-03 & 2.09&   0.7199E-01 & 0.91\\
 7&   0.1751E-03 & 2.06&   0.3718E-01 & 0.95\\
 8&   0.4287E-04 & 2.03&   0.1890E-01 & 0.98\\
 \hline
 &\multicolumn{4}{c}{by $P_2$ elements} \\ \hline
 6&   0.6446E-05 & 2.94&   0.1744E-02 & 1.95\\
 7&   0.8197E-06 & 2.98&   0.4424E-03 & 1.98\\
 8&   0.1033E-06 & 2.99&   0.1113E-03 & 1.99\\
 \hline
 &\multicolumn{4}{c}{by $P_3$ elements} \\ \hline
 6&   0.4457E-07 & 4.02&   0.2293E-04 & 2.97\\
 7&   0.2772E-08 & 4.01&   0.2902E-05 & 2.98\\
 8&   0.1730E-09 & 4.00&   0.3650E-06 & 2.99\\
 \hline
 &\multicolumn{4}{c}{by $P_4$ elements} \\ \hline
 5&   0.2057E-07 & 5.03&   0.4748E-05 & 3.95\\
 6&   0.6344E-09 & 5.02&   0.3009E-06 & 3.98\\
 7&   0.1984E-10 & 5.00&   0.1893E-07 & 3.99\\
 \hline
 &\multicolumn{4}{c}{by $P_5$ elements} \\ \hline
 4&   0.2481E-07 & 6.04&   0.3223E-05 & 4.94\\
 5&   0.3811E-09 & 6.02&   0.1024E-06 & 4.98\\
 6&   0.5938E-11 & 6.00&   0.3225E-08 & 4.99\\
 \hline
    \end{tabular}%
\end{table}%

\end{document}